\documentclass{amsart}
\usepackage[cp1251]{inputenc}
\usepackage[english,russian]{babel}
\usepackage{amsmath}
\usepackage{amssymb}
\usepackage{amsfonts}
\def\udcs{517.956} 
\newtheorem{lemma}{Лемма}
\newtheorem{theorem}{Теорема}

\begin{document}
УДК \udcs \thispagestyle{empty}

\title[Третий потенциал двойного слоя для уравнения Гельмгольца \dots ] {Третий потенциал двойного слоя для обобщенного двуосесимметрического уравнения Гельмгольца }

\author{ Т.Г.ЭРГАШЕВ}
\address{Тухтасин Гуламжанович Эргашев, 
\newline\hphantom{iii} Tuhtasin Gulamjanovich Ergashev,
\newline\hphantom{iii} Ташкентский институт инженеров ирригации и механизации сельского хозяйства,
\newline\hphantom{iii} ул.Кари-Ниязи, 39, 
\newline\hphantom{iii} 100000, г. Ташкент, Узбекистан}
\email{ertuhtasin@mail.ru}

\thanks{\sc Ergashev T.G.
Third Double-Layer Potential for a Generalized Bi-Axially
Symmetric
 Helmholtz Equation}
\thanks{\copyright \ 2017 Эргашев Т.Г.}

\thanks{\it Поступила 1 августа  2017 г.}

\maketitle { \small
\begin{quote}
\noindent{\bf Аннотация. } Потенциал двойного слоя играет важную
роль при решении краевых задач для эллиптических уравнений, при
исследовании которого существенно используются свойства
фундаментальных решений данного уравнения. В настоящее время все
фундаментальные решения обобщенного двуосесимметрического
уравнения Гельмгольца известны, но, несмотря на это, только для
первого из них построена теория потенциала. В данной  работе
исследуется потенциал двойного слоя, соответствующий третьему
фундаментальному решению. Используя свойства
ги\-пер\-гео\-мет\-ри\-чес\-кой функции Аппеля от двух переменных,
до\-ка\-зы\-ва\-ются предельные теоремы и выводятся интегральные
урав\-не\-ния, содержащие в ядре плотность потенциала двойного
слоя.
\medskip

\noindent{\bf Abstract. } {The double-layer potential plays an
important role in solving boundary value problems for elliptic
equations, and in the study of which for a certain
equation, the properties of the fundamental solutions of the given
equation are used. All the fundamental solutions of the
generalized bi-axially symmetric Helmholtz equation were known,
and only for the first one was  constructed the theory of
potential. Here, in this paper, we aim at constructing theory of
double-layer potentials corresponding to the third fundamental
solution. By using some properties of one of Appell's
hypergeometric functions in two variables, we prove limiting
theorems and derive integral equations concerning a denseness of
double-layer potentials.}
\medskip

\noindent{\bf Ключевые слова:}{\,\,обобщенное двуосесимметрическое
уравнение Гельмгольца; формула Грина; фундаментальное решение;
третий потенциал двойного слоя; гипергеометрические функции Аппеля
от двух переменных; интегральные уравнения с плотностью потенциала
двойного слоя в ядре.}
\medskip

 \noindent{\bf Keywords:}{\,\,Generalized bi-axially symmetric Helmholtz equation; Green's formula;\,\,fundamental solution; third double-layer potential;\,\,Appell's hypergeometric functions in two variables;\,\,integral equations concerning a denseness of
double-layer potential.}
\end{quote}
}

\section{Введение}

Многочисленные приложения теории потенциала можно найти в механике жидкости,
эластодинамике, электромагнетизме и акустике. С помощью этой теории
краевые задачи удаётся свести к решению интегральных
уравнений.

Потенциал двойного слоя играет важную роль при решении краевых
задач для эллиптических уравнений. Потому что, метод разделения
переменных и метод функции Грина позволяют получить явное
выражение для решения краевых задач только в случае областей
простейшего вида, а сведение краевых задач при помощи потенциала
двойного слоя к интегральным уравнениям, с одной стороны, удобно
для теоретического исследования вопроса о разрешимости и
единственности краевых задач, с другой стороны, дает возможность
эффективного численного решения краевых задач для областей сложной
формы [1,2].

Применяя метод комплексного анализа (основанный на аналитических
функциях), впервые Гильберт [3] построил интегральное
представление ре\-ше\-ний следующего обобщенного
двуосесимметрического уравнения Гельм\-гольца \frenchspacing
$$H^\lambda_{\alpha,\beta}(u)\equiv u_{xx}+u_{yy}+\frac{2\alpha}{x}u_x+\frac{2\beta}{y}u_y-\lambda^2u=0,\eqno (H^\lambda_{\alpha,\beta})$$
где $\alpha$, $\beta$  и $\lambda \, \--$  постоянные, причем
$0<2\alpha, 2\beta<1$.

Фундаментальные ре\-ше\-ния уравнения $(H^\lambda_{\alpha,\beta})$
найдены  в работе [4]. Когда $\lambda=0$, все четыре
фун\-да\-мен\-таль\-ные решения
$q_i(x,y;x_0,y_0)(i=\overline{1,4})$ уравнения
$H^0_{\alpha,\beta}(u)=0$ можно \,выразить \,с по\-мощью \,
гипергеометрической \,функции Ап\-пеля от двух переменных второго
рода $F_2\left(a,b_1,b_2;c_1,c_2;x,y\right)$, определенной по
формуле [5,6,7]
$$F_2\left(a,b_1,b_2;c_1,c_2;x,y\right)=\sum_{m,n=0}^\infty\frac{(a)_{m+n}(b_1)_m(b_2)_n}{(c_1)_m(c_2)_n m!n!}x^my^n ,$$
где $(a)_n$ ~---~ символ Похгаммера:
$(a)_0=1,(a)_n=a(a+1)(a+2)...(a+n-1),n=1,2,....$

К такому направлению исследований примыкает работа [8], в которой
построены фундаментальные решения  $B$-эллиптических уравнений с
млад\-ши\-ми членами вида
$$u_{xx}+u_{yy}+2\alpha u_x+\frac{2\beta}{y}u_y-\lambda^2u=0.
$$

 В работах [9] и [10] изложена теория потенциала для простейшего
вы\-рож\-даю\-ще\-гося эллиптического уравнения
$H^0_{\alpha,\beta}(u)=0$ при $\alpha=0$ и $\beta=0$,
со\-от\-вет\-ственно.

В [11]  построена  теория потенциала двойного слоя для уравнения
$(H^\lambda_{\alpha,\beta})$ при $\lambda=0$ в области
$$\Omega\subset
R^2_+\left\{(x,y):x>0,y>0\right\} $$ лишь для первого
фундаментального решения $q_1(x,y;x_0,y_0)$.

В настоящей работе мы исследуем потенциал двойного слоя,
соот\-вет\-ст\-вующий третьему фундаментальному решению
$$\,\,\,\,q_3(x,y;x_0,y_0)=\\$$ $$=k_3\left(r^2\right)^{-\alpha+\beta-1}y^{1-2\beta}y_0^{1-2\beta}F_2\left(1+\alpha-\beta;\alpha,1-\beta;2\alpha,2-2\beta; \xi,\eta \right), \eqno (1.1)$$
где
$$k_3=\frac{2^{2+2\alpha-2\beta}}{4\pi}\frac{\Gamma(\alpha)\Gamma(1-\beta)\Gamma(1+\alpha-\beta)}{\Gamma(2\alpha)\Gamma(2-2\beta)}, \eqno (1.2)$$
$$\left. \begin{matrix}
   r^2   \\
   r^2_1   \\
   r^2_2 \\
\end{matrix} \right\}=\left( \begin{matrix}
   x-x_0   \\
   x+x_0   \\
   x-x_0 \\
\end{matrix} \right)^2+\left( \begin{matrix}
   y-y_0   \\
   y-y_0   \\
   y+y_0\\
\end{matrix} \right)^2,\quad \quad  \xi=\frac{r^2-r^2_1}{r^2},\,\,\eta=\frac{r^2-r^2_2}{r^2}.\eqno(1.3)$$

Нетрудно проверить, что функция $q_3(x,y;x_0,y_0)$ обладает
следующими свойствами

$${{\left. \frac{\partial q_3(x,y;x_0,y_0)}{\partial x}
\right|}_{x=0}}=0, \eqno (1.4)$$ $${{\left. q_3(x,y;x_0,y_0)
\right|}_{y=0}}=0. $$

\section{ Формула Грина }

Рассмотрим тождество
$$x^{2\alpha}y^{2\beta}\left[uH^0_{\alpha,\beta}(v)-vH^0_{\alpha,\beta}(u)\right]=$$ $$=\frac{\partial}{\partial x}\left[x^{2\alpha}y^{2\beta}\left(v_xu-vu_x\right)\right]+\frac{\partial}{\partial y}\left[x^{2\alpha}y^{2\beta}\left(v_yu-vu_y\right)\right].$$
Интегрируя обе части последнего тождества по области $\Omega$ ,
расположенной в первой четверти $(x>0,y>0)$ и пользуясь формулой
Остроградского, получим

$$\iint\limits_\Omega
x^{2\alpha}y^{2\beta}\left[uH^0_{\alpha,\beta}(v)-vH^0_{\alpha,\beta}(u)\right]dxdy=$$
$$=\int\limits_S x^{2\alpha}y^{2\beta}u
\left(v_xdy-v_ydx\right)-x^{2\alpha}y^{2\beta}v
\left(u_xdy-u_ydx\right),\eqno (2.1)$$ где $S=\partial\Omega$
~---~ контур области  $\Omega$.

Формула Грина (2.1) выводится  при следующих предположениях:
функции $u(x,y),$ $v(x,y)$ и их частные производные первого
порядка непрерывны в замкнутой области $\overline{\Omega}$ ,
частные производные второго порядка непрерывны внутри $\Omega$  и
интегралы по $\Omega$, содержащие $H^0_{\alpha,\beta}(u)$ и
$H^0_{\alpha,\beta}(v)$, имеют смысл. Если $H^0_{\alpha,\beta}(u)$
и $H^0_{\alpha,\beta}(v)$ не обладают непрерывностью вплоть до
$S$, то это ~---~ несобственные интегралы, которые получаются как
пределы по любой после\-до\-вательности областей $\Omega_n$,
которые содержатся внутри $\Omega$, когда эти области $\Omega_n$
стремятся к $\Omega$, так что всякая точка, находящаяся внутри
$\Omega,$ попадает внутрь областей $\Omega_n, $ начиная с
некоторого номера $n$.

Если $u(x,y)$ и $v(x,y)$ суть решения уравнения
$H^0_{\alpha,\beta}(u)=0$, то из формулы (2.1) имеем
$$\int\limits_Sx^{2\alpha}y^{2\beta}\left(u\frac{\partial v}{\partial n}-v\frac{\partial u}{\partial n}\right)ds=0. \eqno (2.2)$$
Здесь  $$\frac{\partial }{\partial n}=\frac{dy}{\partial
s}\frac{\partial}{\partial x}-\frac{dx}{\partial
s}\frac{\partial}{\partial y} \eqno(2.3)$$ ~---~ оператор
производной по внешней нормали $n$ к кривой  $S$ и $$
\frac{dy}{ds}=cos(n,x),\,\, \frac{dx}{ds}=-cos(n,y)\eqno (2.4)$$
~---~ направляющие косинусы этой нормали.

Полагая в формуле (2.1) $v\equiv1$ и заменяя $u$ на $u^2,$ получим
$$\iint\limits_\Omega
x^{2\alpha}y^{2\beta}\left[u^2_x+u^2_y\right]dxdy=\int\limits_S
x^{2\alpha}y^{2\beta}u\frac{\partial u}{\partial n}ds, $$ где
$u(x,y)$ ~---~ решение уравнения $H^0_{\alpha,\beta}(u)=0$.

Наконец, из формулы (2.2), полагая $v\equiv1$, будем иметь
$$\int\limits_S
x^{2\alpha}y^{2\beta}\frac{\partial u}{\partial n}ds=0, \eqno
(2.5)$$ т.е. интеграл от нормальной  производной решения уравнения
$H^0_{\alpha,\beta}(u)=0$ с ве\-сом $x^{2\alpha}y^{2\beta}$ по
контуру области равен нулю.

\section{Потенциал двойного слоя  $w^{(3)}(x_0,y_0)$}.

Пусть  $\Omega$~---~ область, ограниченная отрезками $(0,a)$ и
$(0,b)$ осей $x$  и  $y$, соответственно, и кривой $\Gamma$ с
концами в точках  $A(a,0)$  и $B(0,b)$, лежащей в первой четверти
$x>0, y>0$ плоскости $R^2$.

Параметрическое уравнение кривой $\Gamma$ пусть будет $x=x(s)$ и
$y=y(s) \,(s\in[0,l]),$ где $s$~---~ длина дуги, отсчитываемая от
точки $B$. Относительно кривой  $\Gamma$ будем предполагать, что:

1) функции $x=x(s)$ и $y=y(s)$ имеют непрерывные производные
 $x'(s)$ и $y'(s)$ на отрезке $[0,l]$, не обращающиеся одновременно в нуль;
вторые производные  $x''(s)$ и $y''(s)$ удовлетворяют условию
Гельдера с показателем $\varepsilon (0<\varepsilon<1)$ на $[0,l]$,
где $l$~---~ длина кривой $\Gamma $;

2) в окрестностях точек $A(a,0)$ и $B(0,b)$ на кривой $\Gamma$
выполняются условия

$$\left| {\frac{dx}{ds}} \right|\le Cy^{1+\varepsilon }\left( s
\right),\, \left| {\frac{dy}{ds}} \right|\le Cx^{1+\varepsilon
}\left( s \right), \eqno (3.1)$$ где $C=const$. Координаты
переменной точки на кривой $\Gamma$ будем обозначать через $(
x,y).$

Рассмотрим интеграл
$$w^{(3)}(x_0,y_0)=\int\limits_0^lx^{2\alpha}y^{2\beta}\mu_3(s)\frac{\partial q_3(x,y;x_0,y_0)}{\partial n}ds, \eqno (3.2)$$
где $\mu_3(s)$~---~ непрерывная функция в промежутке $[0,l],$ а
$q_{3}(x,y;x_0,y_0)$ ~---~ фун\-да\-мен\-таль\-ное решение
уравнения $H^0_{\alpha,\beta}(u)=0,$ определенное по формуле
(1.1).

Интеграл (3.2)  будем называть \textit{третьем потенциалом
двойного слоя с плотностью}  $\mu_3(s)$. Очевидно, что $w^{(3)}(
x_0 ,y_0)$ есть регулярное решение уравнения
$H^0_{\alpha,\beta}(u)=0$ в любой области, лежащей в первой
четверти, не име\-ю\-щей общих точек ни с кривой $\Gamma$ , ни с
осью $x$ , ни с осью $y$. Как и в случае ло\-га\-риф\-ми\-ческого
потенциала, можно показать существование потенциала двойного слоя
(3.2) в точках кривой $\Gamma$ для ограниченной плотности
$\mu_3(s)$.

\begin{lemma}
Справедливы следующие формулы
$$w^{(3)}(x_0,y_0)=\left\{ \begin{matrix}
    j(x_0,y_0)-1, \,\text{если} \,\,\,(x_0,y_0)\in\Omega,   \\
    j(x_0,y_0)-\frac{1}{2},\, \text{если}\,\,\,  (x_0,y_0)\in\Gamma, \\
j(x_0,y_0),\,\,\,\,\,\,\,\,\,\,\,\,\, \text{если}\,\,\,  (x_0,y_0)\notin \overline{\Omega},    \\
\end{matrix} \right.\eqno(3.3)$$
где
\,\,\,\,$\overline{\Omega}:=\Omega\cup\Gamma$;\\

$j(x_0,y_0)=(1-2\beta)k_3y_0^{1-2\beta}\int\limits_0^ax^{2\alpha}\times$
$$\times\left((x-x_0)^2+y_0^2\right)^{-\alpha+\beta-1}F\left(1+\alpha-\beta,\alpha;2\alpha;\frac{-4xx_0}{(x-x_0)^2+y_0^2}\right)dx.
\eqno (3.4) $$ Здесь
$F(a,b;c;z)=\sum\limits_{k=0}^\infty\frac{(a)_k(b)_k}{(c)_kk!}z^k$~---~
гипергеометрическая функция Гаусса.
\end{lemma}

\begin {proof}
\textbf{Случай 1}. Пусть точка $ (x_0,y_0)$ находится внутри
$\Omega.$ Вырежем из области $\Omega$ круг малого радиуса $\rho$ с
центром в точке  $(x_0,y_0)$  и обозначим через $\Omega_\rho$
оставшуюся часть области $\Omega$, а через $C_\rho$ окружность
вырезанного круга. В области $\Omega_\rho$  функция $
q_3(x,y;x_0,y_0)$~---~ регулярное ре\-ше\-ние уравнения
$H^0_{\alpha,\beta}(u)=0$. Используя формулу для производной
гипер\-гео\-мет\-ри\-чес\-кой функции Аппеля [12]
$$ \frac{\partial^{m+n}F_2(a;b_1,b_2;c_1,c_2;x,y)}{\partial x^m\partial y^n}=$$ $$=\frac{(a)_{m+n}(b_1)_m(b_2)_n}{(c_1)_m(c_2)_n}F_2(a+m+n;b_1+m,b_2+n;c_1+m,c_2+n;x,y) \eqno (3.5)$$
имеем
$$\frac{\partial q_3(x,y;x_0,y_0)}{\partial x}=-2(1 +\alpha-\beta){k_3}{\left( {{r^2}} \right)^{-\alpha+\beta-2}}{y^{1-2\beta}}y_0^{1-2\beta}P(x,y;x_0,y_0), \eqno (3.6)$$
где
$$P(x,y;x_0,y_0)=(x-x_0)F_2(1+\alpha-\beta;\alpha,1-\beta;2\alpha,2-2\beta;\xi,\eta)+$$ $$+x_0 F_2(2+\alpha-\beta;1+\alpha,1-\beta;1+2\alpha,2-2\beta;\xi,\eta)+ $$
$$+(x-x_0)\left[\frac{(1+\alpha-\beta)\alpha}{2\alpha}\xi F_2(2+\alpha-\beta;1+\alpha,1-\beta;1+2\alpha,2-2\beta;\xi,\eta)\right.$$
$$\left.+ \frac{1-\beta}{2-2\beta}\eta F_2(2+\alpha-\beta;\alpha,2-\beta;2\alpha,3-2\beta;\xi,\eta)\right].\eqno (3.7)$$

Далее применяя известное соотношение [5]:
$$\frac{{{b_1}}}{{{c_1}}}x{F_2}\left( {a + 1;{b_1} + 1,{b_2};{c_1} + 1,{c_2};x,y} \right) + \frac{{{b_2}}}{{{c_2}}}y{F_2}\left( {a + 1;{b_1},{b_2} + 1;{c_1},{c_2} + 1;x,y}
\right)=$$
   $$= {F_2}\left( {a + 1;{b_1},{b_2};{c_1},{c_2};x,y} \right) - {F_2}\left( {a;{b_1},{b_2};{c_1},{c_2};x,y} \right),
                         $$
к квадратной скобке в (3.7),  получим
$$
  \frac{{\partial {q_3}\left( {x,y;{x_0},{y_0}} \right)}}{{\partial x}}= -2(1 +\alpha-\beta){k_3}{\left( {{r^2}} \right)^{-\alpha+\beta-2}}{y^{1-2\beta}}y_0^{1-2\beta}\times$$
$$\times \left[ x_0{F_2}\left(2+\alpha-\beta;1+\alpha, 1-\beta; 1+2\alpha, 2-2\beta; \xi, \eta \right)+\right.$$ $$\left.+(x-x_0){F_2}\left( {2+\alpha-\beta; \alpha,1-\beta; 2\alpha,2-2\beta;\xi ,\eta }
   \right)\right]. \eqno (3.8) $$

 Аналогично находим
$$\frac{{\partial {q_3}\left(
{x,y;{x_0},{y_0}} \right)}}{{\partial y}} =-2(1+\alpha-\beta){k_3}{\left( {{r^2}}
\right)^{-\alpha+\beta-2}}{y^{1 -
2\beta}}y_0^{1-2\beta}\times$$ $$\times\left[y_0{F_2}\left({2+\alpha-\beta;\alpha,2-\beta;2\alpha,3-2\beta
;\xi ,\eta } \right)+\right.$$
$$+\left.(y-y_0){F_2}\left({2+\alpha-\beta;\alpha,1-\beta;2\alpha,2-2\beta
;\xi ,\eta } \right)\right]+$$
$$+(1 -2\beta){k_3}{\left(
{{r^2}}
\right)^{-\alpha+\beta-1}}{y^{-2\beta}}y_0^{1-2\beta}{F_2}\left(
{1+\alpha-\beta ;\alpha, 1-\beta; 2\alpha, 2-2\beta;\xi ,\eta }
\right). \eqno (3.9)$$

 Пользуясь (3.8) и (3.9), в
силу (1.1),(2.3) и (2.4), найдем
 $$\frac{{\partial {q_3}\left( {x,y;{x_0},{y_0}} \right)}}{{\partial n}}=(1+\alpha-\beta){k_3}{\left( {{r^2}} \right)^{-\alpha+\beta-2}}{y^{-2\beta}}y_0^{1-2\beta}Q\left( {x,y;{x_0},{y_0}} \right), \eqno (3.10)$$
где
$$Q\left( {x,y;{x_0},{y_0}} \right)=-r^2y{F_2}\left( {2+\alpha-\beta;\alpha ,1-\beta; 2\alpha, 2-2\beta;\xi,\eta} \right)\frac{\partial }{{\partial n}}\left[ {\ln {r^2}}\right]- $$
$$-2yy_0{F_2}\left( {2+\alpha-\beta;1+\alpha, 1-\beta; 1+2\alpha, 2-2\beta; \xi,\eta } \right)\frac{{dx}}{{ds}}+$$
$$+2x_0y{F_2}\left({2+\alpha-\beta;\alpha,2-\beta; 2\alpha,3-2\beta;\xi ,\eta } \right)\frac{{dy}}{{ds}}+$$
$$+(1-2\beta) r^2{F_2}\left( {1+\alpha-\beta; \alpha,1-\beta; 2\alpha, 2-2\beta; \xi ,\eta } \right)\frac{{dx}}{{ds}}.$$

Теперь интегрируя нормальную производную  $\frac{{\partial
}}{{\partial n}}{q_3}\left( {x,y;{x_0},{y_0}} \right)$ с весом
${x^{2\alpha }}{y^{2\beta }}$ по границе области $\Omega_\rho$, в
силу (2.5), получим

$$\int\limits_0^a {{x^{2\alpha}}{{\left. {\left[ {{y^{2\beta}}\frac{{\partial {q_3}\left( {x,y;{x_0},{y_0}} \right)}}{{\partial n}}} \right]} \right|}_{y = 0}}dx}+\int\limits_0^lx^{2\alpha}y^{2\beta}\mu_3(s)\frac{\partial q_3(x,y;x_0,y_0)}{\partial n}ds-$$
$$-\mathop {\lim }\limits_{\rho  \to 0} \int\limits_{{C_\rho }} {} {x^{2\alpha }}{y^{2\beta }}\frac{{\partial {q_3}\left( {x,y;{x_0},{y_0}} \right)}}{{\partial n}}ds-\int\limits_0^b {{x^{2\alpha}}{y^{2\beta}}{{\left. {{\frac{{\partial {q_3}\left( {x,y;{x_0},{y_0}} \right)}}{{\partial n}}} } \right|}_{x = 0}}dy}=0.$$

Далее, с учетом (3.2) и (1.4), имеем

$$w_1^{\left( 3 \right)}\left( {{x_0},{y_0}} \right) = \mathop {\lim }\limits_{\rho  \to 0} \int\limits_{{C_\rho }} {} {x^{2\alpha }}{y^{2\beta }}\frac{{\partial {q_3}\left( {x,y;{x_0},{y_0}} \right)}}{{\partial n}}ds+$$ $$+\int\limits_0^a {{x^{2\alpha}}{{\left. {\left[ {{y^{2\beta}}\frac{{\partial {q_3}\left( {x,y;{x_0},{y_0}} \right)}}{{\partial y}}} \right]} \right|}_{y = 0}}dx}.
\eqno (3.11)$$

Подставив (3.10) в (3.11), найдем
$$w_1^{(3)}\left( {{x_0},{y_0}} \right) = {k_3}y_0^{1
-2\beta}\mathop {\lim }\limits_{\rho  \to 0} \left\{
{(1+\alpha-\beta)\left[{ - {J_1} - 2{y_0}{J_2} + 2{x_0}{J_3}}
\right]
+ {J_4}} \right\} + {J_5}, \eqno (3.12)$$ где
$${J_1} = \int\limits_{{C_\rho }} {} x^{2\alpha}y
{\left( {{r^2}} \right)^{-\alpha+\beta-1}}{F_2}\left(
{2+\alpha-\beta; \alpha,1-\beta; 2\alpha,2-2\beta ;\xi ,\eta }
\right)\frac{\partial }{{\partial n}}\left[ {\ln {r^2}}
\right]ds,$$  $${J_2}=\int\limits_{{C_\rho }} {}
x^{2\alpha}y{\left( {{r^2}} \right)^{-\alpha+\beta-2}}{F_2}\left(
{2+\alpha-\beta ;1+\alpha,1-\beta; 1+2\alpha, 2-2\beta ;\xi ,\eta
} \right)\frac{{dx(s)}}{{ds}}ds,$$
$${J_3}=\int\limits_{{C_\rho }} {} x^{2\alpha}y{\left( {{r^2}} \right)^{-\alpha+\beta-2}}{F_2}\left( {2+\alpha
-\beta;\alpha,2-\beta; 2\alpha, 3-2\beta;\xi ,\eta }
\right)\frac{{dy(s)}}{{ds}}ds,$$ $$J_4 =
(1-2\beta)\int\limits_{{C_\rho }} {} x^{2\alpha}{\left( {{r^2}}
\right)^{-\alpha+\beta-1}}{F_2}\left( {1+\alpha-\beta;\alpha
,1-\beta;2\alpha ,2-2\beta ;\xi ,\eta }
\right)\frac{{dx(s)}}{{ds}}ds,$$
$${J_5}= \int\limits_0^a {{x^{2\alpha}}{{\left.
{\left[ {{y^{2\beta}}\frac{{\partial {q_3}\left( {x,y;{x_0},{y_0}}
\right)}}{{\partial y}}} \right]} \right|}_{y = 0}}dx}. $$

Вводя полярные координаты
$$x = {x_0} + \rho \cos \varphi , \,\,y = {y_0} + \rho \sin \varphi \eqno (3.13)$$
в интеграле ${J_1}$, получим
$$
  {J_1} = \int\limits_0^{2\pi } ({x_0} + \rho \cos \varphi )^{2\alpha}{{({y_0} + \rho \sin \varphi )}}\times $$ $$\times{{\left( {{\rho ^2}} \right)}^{-\alpha+\beta-1}}
  {F_2}\left( {2+\alpha-\beta; \alpha ,1-\beta ; 2\alpha ,2-2\beta ;\xi ,\eta }
  \right)d\varphi.
   \eqno      (3.14)$$

Исследуем подынтегральное выражение в (3.14). Применяя
последовательно известные формулы [13]
$${F_2}\left( {a;{b_1},{b_2};{c_1},{c_2};x,y}
\right)=$$ $$ = \sum\limits_{i = 0}^\infty  {} \frac{{{{\left( a
\right)}_i}{{\left( {{b_1}} \right)}_i}{{\left( {{b_2}}
\right)}_i}}}{{{{\left( {{c_1}} \right)}_i}{{\left( {{c_2}}
\right)}_i}i!}}{x^i}{y^i}F\left( {a + i,{b_1} + i;{c_1} + i;x}
\right)F\left( {a + i,{b_2} + i;{c_2} + i;y} \right)$$ и $$F\left(
{a,b;c,x} \right) = {\left( {1 - x} \right)^{ - b}}F\left( {c -
a,b;c,\frac{x}{{x - 1}}} \right),\eqno (3.15)$$ получим
$${F_2}\left( {a;{b_1},{b_2};{c_1},{c_2};x,y} \right) = \frac{(1 -
x)^{-b_1}}{(1-y)^{ {b_2}}}\sum\limits_{i = 0}^\infty  {}
\frac{{{{\left( a \right)}_i}{{\left( {{b_1}} \right)}_i}{{\left(
{{b_2}} \right)}_i}}}{{{{\left( {{c_1}} \right)}_i}{{\left(
{{c_2}} \right)}_i}i!}}{\left( {\frac{x}{{1 - x}}}
\right)^i}{\left( {\frac{y}{{1 - y}}} \right)^i} \times
  $$
   $$ \times F\left( {{c_1} - a,{b_1} + i;{c_1} + i;\frac{x}{{x - 1}}} \right)F\left( {{c_2} - a,{b_2} + i;{c_2} + i;\frac{y}{{y - 1}}} \right).
                         \eqno (3.16)$$
Воспользовавшись теперь формулой  (3.16) гипергеометрическую
функ\-цию Аппеля ${F_2}\left( {2+\alpha-\beta; \alpha ,1-\beta ;
2\alpha ,2-2\beta ;\xi ,\eta }
  \right)$ запишем в виде

$${F_2}\left( {2+\alpha-\beta; \alpha, 1-\beta; 2\alpha
,2-2\beta; \xi ,\eta } \right)=$$
   $$= {\left( {{\rho ^2}} \right)^{1+\alpha-\beta }}{\left( {{\rho ^2} + 4x_0^2 + 4{x_0}\rho \cos \,\varphi } \right)^{-\alpha}}{\left( {{\rho ^2} + 4y_0^2 + 4{y_0}\rho \sin \,\varphi } \right)^{\beta-1 }}{P_{11}},
 \eqno                        (3.17)$$
где $${P_{11}} = \sum\limits_{i = 0}^\infty  {} \frac{{{{\left( {2
+\alpha-\beta} \right)}_i}{{\left( {\alpha } \right)}_i}{{\left(
1-\beta  \right)}_i}}}{{{{\left( {2\alpha } \right)}_i}{{\left(
{2-2\beta } \right)}_i}i!}}\times $$ $$\times {\left(
{\frac{{4x_0^2 + 4{x_0}\rho \cos \,\varphi }}{{{\rho ^2} + 4x_0^2
+ 4{x_0}\rho \cos \,\varphi }}} \right)^i}{\left( {\frac{{4y_0^2 +
4{y_0}\rho \sin \,\varphi }}{{{\rho ^2} + 4y_0^2 + 4{y_0}\rho \sin
\,\varphi }}} \right)^i} \times $$
   $$\times F\left( { \alpha+\beta-2 , \alpha +i;2\alpha+i;\frac{{4x_0^2 + 4{x_0}\rho \cos \,\varphi }}{{{\rho ^2} + 4x_0^2 + 4{x_0}\rho \cos\varphi }}} \right)\times$$ $$\times F\left( {-\alpha-\beta,1-\beta+i;2-2\beta+i;\frac{{4y_0^2 + 4{y_0}\rho \sin\varphi }}{{{\rho ^2} + 4y_0^2 + 4{y_0}\rho \sin \,\varphi }}} \right).$$
Используя известную формулу  для $F\left( {a,b;c;1} \right)$ [14]
$$F\left( {a,b;c;1} \right) = \frac{{\Gamma \left( c \right)\Gamma
\left( {c - a - b} \right)}}{{\Gamma \left( {c - a} \right)\Gamma
\left( {c - b} \right)}},c \ne 0, - 1, - 2,...,\operatorname{Re}
\left( {c - a - b} \right) > 0, \eqno (3.18)$$ получим
$$\mathop {\lim }\limits_{\rho  \to 0}
{P_{11}} = \frac{{\Gamma (2\alpha )\Gamma (2-2\beta )}}{{\Gamma
(2+\alpha-\beta )\Gamma (1-\beta )\Gamma (\alpha )}}. \eqno
(3.19)$$

Таким образом, в силу  (3.14), (3.17) и (3.19), окончательно
получим
$$ - (1+\alpha-\beta){k_3}н_0^{1-2\beta
}\mathop {\lim }\limits_{\rho  \to 0} {J_1}
 =  - 1. \eqno (3.20)$$

Далее, учитывая, что
 $$\mathop {\lim }\limits_{\rho  \to 0} \rho \ln \rho  = 0, \eqno
           (3.21)$$
мы имеем
$$\mathop {\lim }\limits_{\rho  \to 0} {J_2} = \mathop {\lim }\limits_{\rho  \to 0} {J_3} = \mathop {\lim }\limits_{\rho  \to 0} {J_4} = 0.
 \eqno (3.22)$$

Наконец, рассмотрим интеграл ${J_5},$ который, согласно формуле
(3.9), можно привести к виду (3.4), т.е.
$${J_5}=j({x_0},{y_0}).
               \eqno (3.23)$$
Теперь, в силу (3.20)~---~(3.23), из (3.12) следует, что в точке
$\left( {{x_0},{y_0}} \right) \in \Omega $  имеет место равенство
$$w_1^{\left( 3
\right)}\left( {{x_0},{y_0}} \right) = j({x_0},{y_0}) - 1. $$

\textbf{Случай 2}. Пусть теперь точка $\left( {{x_0},{y_0}}\right)
$ совпадает с некоторой точкой $M_0,$ лежащей на кривой $\Gamma.$
Проведем окружность малого радиуса $\rho$ с центром в точке
$\left( {{x_0},{y_0}} \right).$ Эта окружность вырежет часть
${\Gamma_\rho }$   кривой $\Gamma.$ Оставшуюся часть кривой
обозначим через $\Gamma  - {\Gamma_\rho }$. Обозначим через
$C'_\rho$  часть окружности $C_\rho$ , лежащей внутри области
$\Omega $  и рассмотрим область ${\Omega _\rho },$ ограниченную
кривыми $\Gamma - {\Gamma_\rho },$ $C'_\rho$ и отрезками  $\left[
{0,a} \right]$ и $\left[ {0,b} \right]$ осей $x$ и $y$,
соответственно. Тогда имеем
 $$w_1^{\left(3\right)}\left( {{x_0},{y_0}} \right) \equiv \int\limits_0^l {{x^{2\alpha }}{y^{2\beta }}} \frac{{\partial {q_3}\left( {x,y;{x_0},{y_0}} \right)}}{{\partial
  n}}ds= $$
   $$= \mathop {\lim }\limits_{\rho  \to 0} \int\limits_{\Gamma  - {\Gamma _\rho }} {{x^{2\alpha }}{y^{2\beta }}\frac{{\partial {q_3}\left( {x,y;{x_0},{y_0}} \right)}}{{\partial n}}ds}.
\eqno                  (3.24)$$

Так как точка $\left( {{x_0},{y_0}}\right) $ лежит вне этой
области, то в этой области  функция ${q_3}\left( {x,y;{x_0},{y_0}}
\right)$ является регулярным решением уравнения ${H_{\alpha ,\beta
}^0}(u)=0 $ и в силу (2.5) верно равенство
$$
  \int\limits_{\Gamma  - {\Gamma _\rho }} {{x^{2\alpha }}{y^{2\beta }}\frac{{\partial {q_3}\left( {x,y;{x_0},{y_0}} \right)}}{{\partial n}}ds}
   = \int\limits_0^a {{x^{2\alpha}}} {\left. {\left[ {{y^{2\beta}}\frac{{\partial {q_3}\left( {x,y;{x_0},{y_0}} \right)}}{{\partial y}}} \right]} \right|_{y = 0}}dx+$$
   $$+ \int\limits_0^b{x^{2\alpha }} {{y^{2\beta }}} { \left. {\frac{{\partial {q_3}\left( {x,y;{x_0},{y_0}} \right)}}{{\partial x}}}  \right|_{x = 0}}dy + \int \limits_{C_\rho} {} {x^{2\alpha }}{y^{2\beta }}\frac{\partial }{{\partial n}}\left\{ {{q_3}\left( {x,y;{x_0},{y_0}} \right)} \right\}ds.
\eqno                  (3.25)$$

Подставляя (3.25) в (3.24), с учетом (3.23) и (1.4), получим

 $$w_1^{(3)}\left( {{x_0},{y_0}} \right) = j({x_0},{y_0}) + \mathop {\lim }\limits_{\rho  \to 0} \int\limits_{C_\rho} {} {x^{2\alpha }}{y^{2\beta }}\frac{{\partial {q_3}\left( {x,y;{x_0},{y_0}} \right)}}{{\partial n}}ds.
     $$

Вводя снова полярные координаты (3.13) с центром в точке $\left(
{{x_0},{y_0}}\right) $ в ин\-тег\-ра\-ле $$\int\limits_{C_\rho} {}
{x^{2\alpha }}{y^{2\beta }}\frac{\partial }{{\partial n}}\left\{
{{q_3}\left( {x,y;{x_0},{y_0}} \right)} \right\}ds $$ и переходя к
пределу  при ${\rho \rightarrow 0 }$, получим $$\lim \limits_{\rho
\rightarrow 0 } \int\limits_{C_\rho} {} {x^{2\alpha }}{y^{2\beta
}}\frac{\partial }{{\partial n}}\left\{ {{q_3}\left(
{x,y;{x_0},{y_0}} \right)} \right\}ds=-\frac{1}{2}.$$

Таким образом,

 $$w_1^{\left(3 \right)}\left( {{x_0},{y_0}} \right) = j({x_0},{y_0}) -
                     \frac{1}{2}.$$

\textbf{Cлучай 3}. Положим, наконец, что точка $\left(
{{x_0},{y_0}}\right) $ лежит вне области  $\Omega.$ Тогда
${q_3}\left( {x,y;{x_0},{y_0}} \right)$  есть регулярное решение
уравнения ${H_{\alpha ,\beta }^0} (u)=0$ внутри области $\Omega$ с
непрерывными производными всех порядков вплоть до контура
$\Gamma,$ и в силу (2.5)
$$
  w_1^{\left( 3 \right)}\left( {{x_0},{y_0}} \right) \equiv \int\limits_0^l {{x^{2\alpha }}{y^{2\beta }}} \frac{\partial }{{\partial n}}\left\{ {{q_3}\left( {x,y;{x_0},{y_0}} \right)}
  \right\}ds= $$
   $$= \int\limits_0^a {{x^{2\alpha}}} {\left. {\left[ {{y^{2\beta}}\frac{{\partial {q_3}\left( {x,y;{x_0},{y_0}} \right)}}{{\partial y}}} \right]} \right|_{y = 0}}dx = j({x_0},{y_0}).
$$
\end {proof}

\begin{lemma}
Справедливы следующие формулы:
$$w^{(2)}(0,y_0)=\left\{ \begin{matrix}
j(0,y_0)-1, \,\,\,\,\,\,\,\,\,\,\,\,\,\,\text{если} \,\,\,y_0 \in (0,b),  \\
j(0,y_0)-\frac{1}{2},\, \text{если}\,\,\,  y_0=0\,\,{\text{или}}\,\,y_0=b,\\
j(0,y_0),\,\,\,\,\,\,\,\,\,\,\,\,\,\,\,\,\,\,\,\, \text{если}\,\,\,  b < {y_0},    \\
\end{matrix} \right. $$
где
 $$j\left( {0,y_0} \right) = \frac{{1-2\beta}}{1+2\alpha} \left(\frac{a^2}{y_0^2+a^2}\right)^{\frac{1}{2}+\alpha}k_3F\left( \frac{1}{2}+\beta,\frac{1}{2}+\alpha;\frac{3}{2}+\alpha;
  \frac{a^2}{y_0^2+a^2}\right).\eqno (3.26)$$
\end{lemma}

\begin {proof} Сначала исследуем функцию $j(x_0, y_0),$ определенную
фор\-му\-лой (3.4), при $x_0=0$:
$$j(0,y_0)=(1-2\beta)k_3y_0^{1-2\beta} \int\limits_0^ax^{2\alpha}\left(x^2+y_0^2\right)^{-\alpha+\beta-1}dx. $$

Используя известную формулу [14]
$$\int\limits_0^ax^{\lambda-1}\left(x^2+b^2\right)^\nu dx=\frac{1}{\lambda}b^{2\nu}a^{\lambda} F\left(-\nu,\frac{\lambda}{2},\frac{\lambda+2}{2};\frac{-a^2}{b^2}\right), \,\,\,(ab>0,\lambda>0),$$
получим
$$j(0,y_0)=(1-2\beta)k_3a^{1+2\beta}y_0^{-1-2\beta} F\left(\alpha-\beta+1,\frac{1}{2}+\alpha;\frac{3}{2}+\alpha;\frac{-a^2}{y_0^2}\right). \eqno (3.27) $$

Далее, воспользовавшись формулой (3.15) получим  функцию
$j(0,y_0),$ оп\-ре\-де\-лен\-ную формулой (3.26). Учитывая
известную формулу (3.18) для $F(a,b;c;1)$ и значение $k_3$  из
формулы (1.2), из (3.26) легко следует, что $j(0,0)=1.$

Пусть теперь точка $(x_0,y_0)$ находится на оси $y$  и пусть в
первом случае будет $y_0\in(0,b)$. Проведем прямую $x=h$
($h>0$~---~ достаточно мало) и рассмотрим область $\Omega_h$,
которая есть часть области $\Omega$, лежащая справа от прямой
$x=h$. Применяя формулу (2.5), получим
               $$w_1^{\left(
3\right)}\left( 0,y_0 \right) = J_6+J_7, \eqno (3.28) $$
где
 $$J_6=\mathop {\lim }\limits_{h \to 0}
\int\limits_h^a {} {\left. {{x^{2\alpha }}{y^{2\beta
}}\frac{{\partial {q_3}\left( {x,y;0,y_0} \right)}}{{\partial y}}}
\right|_{y = 0}}dx,$$
$$ J_7= \mathop {\lim }\limits_{h \to 0}
\int\limits_0^{{y_1}} {} {\left. {{y^{2\beta}}{x^{2\alpha
}}\frac{{\partial {q_3}\left( {x,y;0,y_0} \right)}}{{\partial x}}}
\right|_{x = h}}dx.$$ Здесь $y_1$~---~ ордината точки пересечения
кривой $\Gamma$ с прямой $x=h$.

Нетрудно заметить, что
  $$ J_6=j(0,y_0). \eqno (3.29)$$
Теперь рассмотрим второе слагаемое в (3.28), которое, в силу (3.8), принимает вид
$$J_7=-2(1-\alpha-\beta){k_3}y_0^{1-2\beta}J_8, \eqno (3.30)$$
где
 $$J_8=h^{1+2\alpha}\int\limits_0^{{y_1}} {} y\frac{{F\left( 2+ \alpha-\beta,1-\beta;2-2\beta;-\frac{4yy_0}{{(y-y_0)^2 + {h^2}}} \right)}}{{{{\left[ {{{\left( {y - {y_0}} \right)}^2} + {h^2}} \right]}^{2 +\alpha-\beta
}}}}dy. $$

Преобразуем $J_8$. Воспользовавшись формулой (3.15),
получим
          $$J_8=h^{1+2\alpha}\int\limits_0^{{y_1}} {} y\frac{{F\left( {-\alpha-\beta,1-\beta; 2-2\beta;\frac{{4y{y_0}}}{{{{\left( {y + {y_0}} \right)}^2} + {h^2}}}} \right)}}{{{{\left[ {{{\left( {y - {y_0}} \right)}^2} + {h^2}} \right]}^{1 + \alpha }}{{\left[ {{{\left( {y + {y_0}} \right)}^2} + {h^2}} \right]}^{1 - \beta
}}}}dx,$$ Теперь вместо $y$  введем новую переменную
интегрирования $y=y_0+ht.$ Со\-вер\-шая замену переменных, получим
              $$J_8(h,y_0)= \int\limits_{l_1}^{l_2}\left( {{y_0} + ht} \right)\frac{{F\left( {-\alpha-\beta,1-\beta; 2-2\beta,\frac{{4{y_0}\left( {{y_0} + ht} \right)}}{{{{\left( {2{y_0} + ht} \right)}^2} + {h^2}}}} \right)}}{{{{\left( {1 + {t^2}} \right)}^{\alpha+1}}{{\left[ {{{\left( {2{y_0} + ht} \right)}^2} + {h^2}} \right]}^{1-\beta}}}}dt,
\eqno  (3.31)$$ где
    $${l_1} =  - \frac{{{y_0}}}{h},      {l_2} = \frac{{{y_1} - {y_0}}}{h}.$$
Принимая во внимание, что
$$\mathop {\lim }\limits_{h\to0}F\left({-\alpha-\beta,1-\beta;2-2\beta,\frac{{4{y_0}\left( {{y_0} + ht} \right)}}{{{{\left( {2{y_0} + ht} \right)}^2} + {h^2}}}} \right)= $$ $$= F\left( {-\alpha-\beta,1-\beta;2-2\beta;1} \right)=\frac{{\Gamma \left( {2-2\beta} \right)\Gamma \left( {1 + \alpha} \right)}}{{\Gamma \left( {2+\alpha-\beta} \right)\Gamma \left({1-\beta}
\right)}}$$ и
$$\int\limits_{-\infty}^{+\infty} {} \frac{{dt}}{{{{\left( {1 + {t^2}} \right)}^{\alpha+1}}}} = \frac{{\pi \Gamma ( {2\alpha} )}}{{{2^{2\alpha-1}}\alpha{\Gamma^2}(\alpha )}},$$
из (3.29)~---~(3.31) находим
$$w_1^{( 3)}\left( 0,y_0 \right) = j(0,y_0) - 1. $$

Остальные три случая, когда $y_0=0,$ $y_0=b$ и $y_0>b$,
доказываются ана\-ло\-гич\-но первому случаю.

\end {proof}
\begin{lemma} Для любых точек  $(x,y)$ и $(x_0,y_0)\in R_+^2$  при $x\neq x_0$ и $y\neq y_0$ спра\-вед\-ли\-во
не\-ра\-венство
        $$\left| {{q_3}\left( {x,y;{x_0},{y_0}} \right)} \right| \leqslant \frac{{\Gamma(\alpha)\Gamma(1-\beta)}}{{\pi \Gamma(1+\alpha-\beta)}}\frac{{{4^{\alpha-\beta}}{y^{1-2\beta}}y_0^{1 - 2\beta}}}{{{{\left( {r_1^2} \right)}^{\alpha}}{{\left( {r_2^2} \right)}^{1-\beta} }}}\times$$ $$\times F\left[ {\alpha,1-\beta;1+\alpha-\beta;\left( {1 - \frac{{{r^2}}}{{r_1^2}}} \right)\left( {1 - \frac{{{r^2}}}{{r_2^2}}} \right)} \right],
    \eqno (3.32)$$
где $\alpha$  и $\beta$ ~---~ действительные числа, причем
$0<2\alpha,2\beta<1$,  а  $r$, $r_1$ и $r_2 $~---~
вы\-ра\-же\-ния, определенные в (1.3).
\end{lemma}

\begin {proof} Из (3.16) следует, что
 $${q_3}\left( {x,y;{x_0},{y_0}} \right)={k_3}{y^{1-2\beta}}y_0^{1-2\beta}{\left( {r_1^2} \right)^{-\alpha}}{\left( {r_2^2} \right)^{\beta-1 }}\times$$ $$\times \sum\limits_{i = 0}^\infty  {} \frac{{{{\left( {1 + \alpha-\beta } \right)}_i}{{\left( {\alpha } \right)}_i}{{\left(1-\beta \right)}_i}}}{{{{\left( {2\alpha} \right)}_i}{{\left( {2-2\beta } \right)}_i}i!}}{\left( {1 - \frac{{{r^2}}}{{r_1^2}}} \right)^i}{\left( {1 - \frac{{{r^2}}}{{r_2^2}}} \right)^i} \times
  $$
   $$\times F\left( {\alpha+\beta-1 ,\alpha+i;2\alpha+i;1-\frac{{{r^2}}}{{r_1^2}}} \right)\times$$ $$\times F\left({1-\alpha-\beta,1-\beta+i;2-2\beta+i;1-\frac{{{r^2}}}{{r_2^2}}} \right),
\eqno (3.33)$$ Теперь, ввиду следующих неравенств:

$$F\left({\alpha+\beta-1,
\alpha+i;2-2\alpha+i;1-\frac{{{r^2}}}{{r_1^2}}} \right) \leqslant
\frac{{{{(2\alpha)}_i}\Gamma (2\alpha)\Gamma (1-\beta
)}}{{{{(1+\alpha-\beta)}_i}\Gamma(1+\alpha-\beta)\Gamma(\alpha)}}$$
и $$F\left( {1-\alpha-\beta,1-\beta+i;2-2\beta+i;1-
\frac{{{r^2}}}{{r_2^2}}} \right) \leqslant \frac{{{{(2-2\beta
)}_i}\Gamma (2-2\beta)\Gamma
(\alpha)}}{{{{(1+\alpha-\beta)}_i}\Gamma(1+\alpha-\beta)\Gamma
(1-\beta)}},$$ из (3.33)  следует неравенство (3.32).
\end {proof}

В силу известной формулы [6] $$F\left( {a,b;a + b;z} \right) = -
\frac{{\Gamma \left( {a + b} \right)}}{{\Gamma \left( a
\right)\Gamma \left( b \right)}}F\left( {a,b;1;1 - z} \right)\ln
\left( {1 - z} \right) + $$ $$ + \frac{{\Gamma \left( {a + b}
\right)}}{{{\Gamma ^2}\left( a \right){\Gamma ^2}\left( b
\right)}}\sum\limits_{j = 0}^\infty  {} \frac{{\Gamma \left( {a +
j} \right)\Gamma \left( {b + j} \right)}}{{{{\left( {j!}
\right)}^2}}}\left[ {2\psi \left( {1 + j} \right) - \psi \left( {a
+ j} \right) - \psi \left( {b + j} \right)} \right]{\left( {1 - z}
\right)^j},$$ $\left( { - \pi  < \arg \,\left( {1 - z} \right) <
\pi ,\,\,a,b \ne 0, - 1, - 2,...} \right)$, из  (3.32) следует
[4], что  функция $ {q_3}\left( {x,y;{x_0},{y_0}} \right)$  имеет
логарифмическую особенность при $r=0$.

\begin{lemma} Если кривая $ \Gamma$ удовлетворяет перечисленным выше
условиям, то $$\int\limits_\Gamma  {} {x^{2\alpha }}{y^{2\beta
}}\left| {\frac{{\partial {q_3}\left( {x,y;{x_0},{y_0}}
\right)}}{{\partial n}}} \right|ds \leqslant {C_1}, \eqno (3.34)$$
где $C_1$ ~---~ постоянная.
\end{lemma}
\begin {proof} Неравенство (3.34) следует из условий (3.1) и формулы
(3.10).
\end {proof}
Формулы (3.3) показывают, что при $\mu_3(s)\equiv 1$  потенциал
двойного слоя ис\-пы\-ты\-вает разрыв непрерывности, когда точка
$(x,y)$ пересекает кривую $\Gamma$. В случае произвольной
непрерывной плотности $\mu_3(s)$ имеет место

\begin{theorem}
Потенциал двойного слоя  $w^{(3)}(x_0,y_0)$ имеет пределы при
стремлении точки $(x_0,y_0)$ к точке $(x(s),y(s))$ кривой $\Gamma$
извне или изнутри. Если предел значений $w^{(3)}_i(x_0,y_0)$
изнутри обозначить через $w^{(3)}(s)$, а предел извне ~---~ через
$w^{(3)}_e(s)$, то имеют место формулы
 $$w_i^{(3)}(t) =-\frac{1}{2}{\mu
_3}\left( t \right) + \int\limits_0^l {} {\mu _3}\left( s
\right){K_3}\left( {s,t} \right)ds $$ и $$w_e^{(3 )}(t) =
\frac{1}{2}{\mu _3}(t) + \int\limits_0^l {\mu _3}(s){K_3}(s,t)ds,
$$ где
$${K_3}(s,t) = {[x(s)]^{2\alpha }}{[y(s)]^{2\beta
}}\frac{\partial }{{\partial n}}\left\{ {{q_3}\left[ {x\left( s
\right),y\left( s \right);{x_0}(t),{y_0}(t)} \right]} \right\}.$$
\end{theorem}

\begin {proof} Справедливость утверждений теоремы 1 следует из лемм 1~---~4.
\end {proof}
Функция
$$w_0^{(3)}(s) = \int\limits_0^l {} {\mu _3}(t){K_3}(s,t)dt$$ непрерывна при $0\leq s \leq l$ , что следует из хода доказательства теоремы 1. В силу результатов теоремы 1 и непрерывности  функций $w_0^3(s)$  и $\mu_3(s)$  при  $0\leq s \leq l$, следует, что  потенциал двойного слоя $w^{(3)}(x_0,y_0)$  есть функция непрерывная внутри области $\Omega$  вплоть до кривой
$\Gamma.$ Точно также $w^{(3)}(x_0,y_0)$  непрерывна вне области
$D$ вплоть до кривой $\Gamma$.

В заключении отметим, что полученные в настоящем сообщении
результаты играют важную роль при ре\-ше\-нии краевых задач для
уравнения ${H_{\alpha,\beta}^0}(u)=0$. При этом решение
пос\-тав\-лен\-ной задачи ищется в виде третьего потенциала
двойного слоя (3.2) с неизвестной плотностью $\mu_3(s)$, для
определения которой используется известная теория интегральных
уравнений Фредгольма второго рода.

\end{document}